\documentclass[11pt]{amsart}
\usepackage[margin=3.5cm]{geometry}

\usepackage{amsmath,amsthm,amssymb,multicol}
\usepackage{blindtext}
\usepackage{hyperref}
\usepackage{cleveref}
\usepackage{diagbox}
\usepackage{soul}
\usepackage[dvipsnames]{xcolor}

\newcommand{\tL}{\mathtt 1}
\newcommand{\tO}{\mathtt 0}

\theoremstyle{definition}
\newtheorem{theorem}{Theorem}
\newtheorem{Def}{Definition}

\newtheorem{lem}[theorem]{Lemma}
\newtheorem{pb}{Problem}

\author{Pierre Popoli}

\title[Maximum order complexity along polynomials]{On the maximum order complexity of Thue--Morse and Rudin--Shapiro sequences along polynomial values}

\address{
Institut \'Elie Cartan de Lorraine,
Universit\'e de Lorraine,
Vand\oe uvre-l\`es-Nancy, France}


\begin{document}

\begin{abstract}
Both the Thue--Morse and Rudin--Shapiro sequences are not suitable sequences for cryptography since their expansion complexity is small and their correlation measure of order 2 is large. These facts imply that these sequences are highly predictable despite the fact that they have a large maximum order complexity. Sun and Winterhof (2019) showed that the Thue--Morse sequence along squares keeps a large maximum order complexity. Since, by Christol's theorem, the expansion complexity of this rarefied sequence is no longer bounded, this provides a potentially better candidate for cryptographic applications. Similar results are known for the Rudin--Shapiro sequence and more general pattern sequences. In this paper we generalize these results to any polynomial subsequence (instead of squares) and thereby answer an open problem of Sun and Winterhof. We conclude this paper by some open problems.
\end{abstract}

\maketitle

\section{Introduction}

Pseudorandomness, i.e. the study of phenomena related to randomness for deterministic objects, has grown to a large and important subject in number theory and cryptography. In recent years, research focused in particular on automatic sequences (e.g. Thue--Morse sequence, Rudin--Shapiro sequence), i.e. sequences that are generated by a deterministic finite automaton. Such sequences are easy to generate and ``regular'' in some sense, but their behavior changes radically when the sequence is rarefied along a subsequence so that the rarefied sequence shows pseudorandom behavior. The aim of the present article is to study pseudorandomness in the context of such polynomially rarefied automatic sequences. 

To begin with, we first introduce various measures for pseudorandomness and cite the results that are known in the context of classical automatic sequences.

\begin{Def}[Maximum order complexity]
Let $N$ be a positive integer with $N \geq 2$, and $\mathcal{S}=\left(s_n \right)_{n\geq 0}$ be a sequence over $\{0,1\}$ with $(s_0,\ldots,s_{N-2})\neq (a,\ldots,a)$ for $a=0$ or $1$. The  \textit{$N$th maximum order complexity} $M(\mathcal{S},N)$ is the smallest positive integer $M$ such that there is a polynomial $f(x_1,\ldots,x_M)$ with \begin{align*}
s_{i+M}=f(s_i,\ldots,s_{i+M-1}), \quad 0\leq i\leq N-M-1.
\end{align*}
If $s_i=a$ for $i=0,\ldots,N-2$, we define $M(\mathcal{S},N)=0$ if $s_{N-1}=a$ and $M(\mathcal{S},N)=N-1$ else. 
\end{Def}

A sequence with small maximum order complexity cannot be used in cryptography, since the sequence can be constructed from relatively short blocks of consecutive terms. However, a sequence with large maximum order complexity, is not automatically adapted in cryptography. It can still be very predictable as we will state later for the Thue--Morse sequence. 

Diem~\cite{diem2012} introduced the expansion complexity of a sequence as follows.

\begin{Def}[Expansion complexity]
Let $N$ be a positive integer, $\mathcal{S}=\left(s_n \right)_{n\geq 0}$ be a sequence over $\{0,1\}$ and $G(x)$ its generating function defined by \begin{align*}
G(x)=\sum \limits_{i\geq 0}s_ix^i.
\end{align*} The \textit{$N$th expansion complexity} $E(\mathcal{S},N)$ is defined as the least total degree of a nonzero polynomial $h(x,y) \in \mathbb{F}_2[x,y]$ with \begin{align*}
h(x,G(x)) \equiv 0 \pmod{x^N},
\end{align*} 
if $s_0,\ldots,s_{N-1}$ are not all equal to $0$, and $E(\mathcal{S},N)=0$ otherwise. 
\end{Def}

Similarly to the maximum order complexity, a sequence with small $N$th expansion complexity is predictable. By Christol's theorem (see~\cite{christol1979}) automatic sequences over $\mathbb{F}_p$ are characterized by \begin{align*}
\sup_{N\geq 1}E(\mathcal{S},N) < \infty.
\end{align*}
This indicates that automatic sequences are not pseudorandom and may be considered as cryptographically weak. 

Mauduit and S\'{a}rk\"{o}zy~\cite{mauduitsarkozy1997} introduced the correlation measure.
\begin{Def}[Correlation measure of order $2$]
Let $N$ be a positive integer, $\mathcal{S}=(s_n)_{n\geq 0}$ be a sequence over $\{0,1\}$. The \textit{N-th correlation measure of order~2} of $\mathcal{S}$ is \begin{align*}
C_2\left(\mathcal{S},N\right)=\max_{M,d_1,d_2}\left|\sum \limits_{0\leq n \leq M}(-1)^{s_{n+d_1}+s_{n+d_2}}\right|,
\end{align*}
where the maximum is taken over all $M$, $d_1$ and $d_2$ such that $0\leq d_1<d_2$ and $d_2+M<N$. 
\end{Def}

For a random sequence, the correlation measure of order $2$ is of order of $(N\log(N/2))^{1/2}$ (see~\cite{alongkohayakawamauduitmoreira2007}). 

We introduce the symbolic complexity as an other measure of pseudorandomness. 
\begin{Def}[Symbolic complexity]
The \textit{symbolic complexity}, or \textit{subword complexity}, of a sequence $\mathcal{S}$ over $\{0,1\}$ is the function $p_{\mathcal{S}}$ defined for every positive integer $k$ by \begin{align*}
p_{\mathcal{S}}(k)=\text{Card}\{(b_0,\ldots,b_{k-1}) \in \{0,1\}^k : \exists i, u(i)=b_0,\ldots,u(i+k-1)=b_{k-1} \}.
\end{align*}
\end{Def}

A sequence $\mathcal{S}$ over $\{0,1\}$ is \textit{normal} if for every $k\geq 1$ and any $(b_0,\ldots,b_{k-1}) \in \{0,1\}^k$, we have \begin{align*}
\lim_{N \to \infty } \frac{1}{N}\,\text{Card}\{i<N:\; u(i)=b_0,\ldots,u(i+k-1)=b_{k-1} \}=\frac{1}{2^k}.
\end{align*}
For a normal sequence, each block of length $k$ appears and each block appears with the same frequency. A ``good'' pseudorandom sequence should have a large symbolic complexity.

We now look at these complexity measures for the Thue--Morse sequence. One possible definition of this emblematic sequence is as follows.

\begin{Def}[Thue--Morse sequence]
For an integer $n\geq 0$, we write $n=\sum_{i \geq 0}\varepsilon_i2^i$ with $\varepsilon_i \in \{0,1\}$ for all $i$ and $(n)_2=\cdots\varepsilon_1 \varepsilon_0$. The binary sum-of-digits of $n$ equals $s_1(n)=\sum_{ i \geq 0} \varepsilon_i$. The \textit{Thue--Morse sequence} $\mathcal{T}=(t(n))_{n\geq 0}$ is defined by $t(n)=s_1(n) \bmod 2$. 
\end{Def}

Note that the index in the sum-of-digits function relates to the length of the $\tL$-pattern that we consider (the length is 1 here; we will consider $k$ consecutive $\tL$'s for \textit{pattern sequences}). In what follows, we use Vinogradov's notation $f\ll g$ if there is a constant $c>0$ such that $f\leq cg$. 
 
Sun and Winterhof~\cite{sunwinterhof2019} showed that the Thue--Morse sequence has large maximum order complexity, $M(\mathcal{T},N)\gg N$. Mauduit and S\'ark\"ozy~\cite[Theorem 2]{mauduitsarkozy1998} showed that the correlation measure of order 2 of the Thue--Morse sequence is large, $C_2\left(\mathcal{T},N\right)\gg N$. On the other hand, it is well-known that $E(\mathcal{T},N) \leq 5$ for all $N$ since $h(x,y)=(x+1)^3y^2+(x+1)^2y+x$ satisfies $h(x,G(x))=0$, where $G(x)$ is the generating function of $\mathcal{T}$. Also, its symbolic complexity is small, $p_{\mathcal{T}}(k)\ll k$ (see~\cite[Corollary 10.3.2]{alloucheshallit2003} for a general result for all automatic sequences). The small upper bounds of the expansion complexity and of the symbolic complexity imply that the Thue--Morse sequence is far from being a pseudorandom sequence with respect to these measures.

Several of the mentioned results also hold true for more general pattern sequences (also called \textit{Rudin--Shapiro sequences of degree $k$}, see~\cite{mauduitrivat2018}). For the sake of shortness, we refer to them as \textit{pattern sequences}, as they were called by Sun and Winterhof~\cite{sunwinterhof2019}.

\begin{Def}[Pattern sequences]
Let $k\geq1$. Denote by $P_k=\tL \cdots \tL \in \mathbb{F}_2^k$ the all $\tL$-pattern of length $k$ and by $s_k(n)$ the number of occurrences of $P_k$ in the binary digital representation of $n$. The $k$-\textit{pattern sequence} $\mathcal{P}_k=(p_k(n))_{n\geq 0}$ (or, for short, \textit{pattern sequence}) is defined by 
\begin{align*}
  p_k(n)= s_k(n) \bmod 2.
\end{align*}
\end{Def}

For $k=1$ we get the Thue--Morse sequence $\mathcal{T}=\mathcal{P}_1$ and for $k=2$ we get the Rudin--Shapiro (or Golay--Rudin--Shapiro) sequence $\mathcal{R}=\mathcal{P}_2=(r(n))_n$. As the Thue--Morse sequence, the pattern sequence $\mathcal{P}_k$ is $2$-automatic and has a large maximum order complexity (see~\cite[Theorem 2]{sunwinterhof2019}). Its expansion complexity satisfies $E(\mathcal{P}_k,N)\leq 2^k+3$ for $N\geq 1$ since $h(x,y)=(x+1)^{2^{k+1}+1}y^2+(x+1)^{2^{k}}y+x^{2^{k}-1}$ satisfies $h(x,G(x))=0$ with $G(x)$ the generating function of $\mathcal{P}_k$ (see~\cite{sunwinterhof2019bis}). M\'erai and Winterhof~\cite[Corollary 4]{meraiwinterhof2018} showed that the correlation of order $2$ for pattern sequences is still large. However, since pattern sequences are still automatic, their symbolic complexity is linear~\cite[Corollary 10.3.2]{alloucheshallit2003}. Therefore, the pattern sequences are not pseudorandom with respect to each of the defined measures.

The behavior of these sequences regarding the defined pseudorandomness measures changes when these sequences are rarefied along specific subsequences. Sun and Winterhof~\cite{sunwinterhof2019bis} showed that the maximum order complexity of the Thue--Morse sequence and pattern sequences along squares remains large. Note that the largest possible order of magnitude of $M(\mathcal{S},N)$ is $N$, while the expected value of $M(\mathcal{S},N)$ is $\log N$ (see~\cite{sunwinterhof2019bis}).

\begin{theorem}[\cite{sunwinterhof2019bis}, Theorem 1]\label{TheoremSunWinterhof1}
Let $\mathcal{T}'=(t(n^2))_n$ be the subsequence of the Thue--Morse sequence along squares. Then the $N$th maximum order complexity of $\mathcal{T}'$ satisfies \begin{align*}
M(\mathcal{T}',N) \geq \sqrt{\frac{2N}{5}}\, , \qquad N\geq 21.
\end{align*}
\end{theorem}

\begin{theorem}[\cite{sunwinterhof2019bis}, Theorem 2]
For $k\geq 2$ let $\mathcal{P}_k'=(p_k(n^2))_n$ be the subsequence of  $\mathcal{P}_k$ along squares. Then the $N$th maximum order complexity of $\mathcal{P}_k'$ satisfies \begin{align*}
M(\mathcal{P}_k',N) \geq \sqrt{\frac{N}{8}}, \qquad N\geq 2^{2k+2}.
\end{align*}
\end{theorem}

Drmota, Mauduit and Rivat~\cite{drmotamauduitrivat2019} showed that $\mathcal{T}'$ is a normal sequence, and M\" ullner~\cite{mullner2018} showed that $\mathcal{R}'$ and more general pattern sequences along squares are normal, too. These statements mean that $\mathcal{T'}$ and $\mathcal{R'}$ might be better candidates for cryptographic applications as the inherent weaknesses for automatic sequences disappear. 

By~\cite[Theorem 6.10.1]{alloucheshallit2003}, $\mathcal{T}'=\mathcal{P}_1'$ is no longer automatic and so \begin{align*}
\sup_{N\geq 1}E(\mathcal{T}',N)=+\infty.
\end{align*}
By \cite{mullner2018} and Christol's theorem we also have for $\mathcal{R}'=\mathcal{P}_2'$,
 \begin{align*}
\sup_{N\geq 1}E(\mathcal{R}',N)=+\infty.
\end{align*}
Passing from the subsequence of squares to more general polynomials seems to be a natural question. More specifically, Sun and Winterhof (Problem 4 in~\cite{sunwinterhof2019bis}) posed the problem to extend their results to this more general context.

In this paper we provide an answer to their problem. 
\begin{theorem}\label{MainTheorem1}
Let $d \geq 2$ and $P(X) \in \mathbb{Z}[X]$ be a monic polynomial of degree $d$ with $P(\mathbb{N})\subset \mathbb{N}$. Let $\mathcal{T}_P=(t(P(n)))_n$ be the subsequence of the Thue--Morse sequence along the polynomial subsequence $(P(n))_n$. Then $\mathcal{T}_P$ satisfies \begin{align*}
M(\mathcal{T}_P,N) \gg N^{1/d},
\end{align*}
where the implied constant only depends on $P$. 
\end{theorem}

We recover the same bound for pattern sequences, too.
\begin{theorem}\label{MainTheorem2}
Let $d \geq 2$ and $P(X) \in \mathbb{Z}[X]$ be a monic polynomial of degree $d$ with $P(\mathbb{N})\subset \mathbb{N}$. Let $\mathcal{P}_{k,P}=(p_k(P(n)))_n$ be the subsequence of the pattern sequence along the polynomial subsequence $(P(n))_n$. Then $\mathcal{P}_{k,P}$ satisfies \begin{align*}
M(\mathcal{P}_{k,P},N) \gg N^{1/d},
\end{align*}
where the implied constant only depends on $P$ and $k$. 
\end{theorem}

It is possible to generalize our theorem for the case of integer-valued polynomials with rationals coefficients, we leave this rather straightforward extension to the interested reader (the proof runs along the same lines). We remark, however, that our construction crucially depends on the fact that the leading coefficient of the polynomial equals one. 

The paper is structured as follows. In Section 2, we generalize \Cref{TheoremSunWinterhof1} to any polynomial subsequence in place of the subsequence of squares and in Section 3 we establish the result for any pattern sequence. This answers a question posed by Sun and Winterhof (Problem 4 in~\cite{sunwinterhof2019bis}). We finish the paper with a list of open problems in Section 4.

\section{Thue--Morse sequence}

The Thue--Morse sequence along arithmetic progressions is $2$-automatic~\cite[Theorem 6.8.1]{alloucheshallit2003}. By~\cite[Theorem 6.10.1]{alloucheshallit2003}, the Thue--Morse sequence along polynomial subsequences is non-automatic if and only if the polynomial is at least of degree 2.  The problem raised by Sun and Winterhof~\cite{sunwinterhof2019bis} is to know whether there still holds a result such as \Cref{TheoremSunWinterhof1} for general polynomial subsequences.

The following trivial identity will be essential for our general proof. Let $a,b$ be positive integers and $0\leq b < 2^r$, then we have \begin{align}\label{trivial1}
s_1(a2^r+b)=s_1(a)+s_1(b).
\end{align}
If we have such $a$ and $b$ we say that the sum is \textit{non-interfering}. The proof of our main result is based both on non-interfering sums and on carry propagation.

\begin{lem}\label{MainLemmaThueMorse}
Let $d\geq2$ and $P(X) \in \mathbb{Z}[X]$ with $P(X)=X^d+\alpha_{d-1}X^{d-1}+\cdots+\alpha_1X+\alpha_0$ such that $P(\mathbb{N}) \subset \mathbb{N}$ and all $\alpha_i\geq 0$. Put $\alpha_{\max}=\max(\alpha_i)$. Then there exists a positive integer $l_0(P)$ such that for all $l>l_0(P)$ the following two properties hold:
\begin{enumerate}
\item[(i)] For all $1\leq n < \frac{1}{2(2\alpha_{\max})^{1/d}}\;2^l$ and for all $r\geq 1$, $$t(P(n+2^{dl}))=t(P(n+2^{dl+r})).$$ 
\item[(ii)] There are nonnegative integers $y$ and $r$ depending only on $P,$ such that \begin{align*}
t(P(1+y2^l+2^{dl}))\neq t(P(1+y2^l+2^{dl+r})).
\end{align*}
\end{enumerate}
\end{lem}

\begin{proof} Set $\alpha_d=1$. For the first part we write
\begin{align} 
P(n+2^{dl})&=\sum \limits_{0\leq j \leq d}\alpha_j(n+2^{dl})^j, \nonumber\\
&=\sum \limits_{0\leq i\leq d} \left( \sum \limits_{i\leq j \leq d} \binom{j}{i}\alpha_jn^{j-i} \right)2^{idl}.\label{noninterferingsum}
\end{align}
Set $\beta_i=\sum \limits_{i\leq j \leq d} \binom{j}{i}\alpha_jn^{j-i}$. We note that for all $0 \leq i \leq d$, 
\begin{align*}
\beta_i \leq \alpha_{\max}\, n^{d-i}\sum \limits_{i\leq j\leq d} \binom{j}{i}\leq \alpha_{\max}\,n^d\binom{d+1}{i+1}\leq \alpha_{\max}\,n^d2^{d+1}.
\end{align*}
There exists a positive integer $l_0(P)$ such that for all $l>l_0(P)$ and all integers $n$ with
\begin{align}\label{condition}
1\leq n < \frac{1}{2(2\alpha_{\max})^{1/d}}\,2^l,
\end{align}
we have $\beta_i <2^{dl}$.
Thus, for $l>l_0(P)$ and all $n$ with~\eqref{condition} the sum~\eqref{noninterferingsum} is non-interfering, thus we get for all $r\geq 1$, 
$$t(P(n+2^{dl}))=\sum \limits_{0\leq i\leq d}t(\beta_i)=t(P(n+2^{dl+r})),$$
which shows property (i).
As for the second part, we write
\begin{align}\label{interferingsum}
P(1+y2^l+2^{dl})&=\sum \limits_{0\leq i\leq d} \left( \sum \limits_{i\leq j \leq d} \binom{j}{i}\alpha_j(1+y2^l)^{j-i} \right)2^{idl}.
\end{align}
We regroup terms by powers of $2$ and check for possible interferences. The general term is $ 2^{idl+l(j-i)}$, for $ 0\leq i \leq  j \leq d$, whereas the coefficients depend only on $P$ and $y$. We represent the general terms in the following table: 

$$\begin{tabular}{|c|c|c|c|c|c|c|}
\hline
\backslashbox{$i$}{$j$} & $0$ & $1$ & $2$ & $\cdots$ & $d$ \\ \hline
$0$ & $2^0$ & $2^{l}$ & $2^{2l}$ &  & $2^{dl}$ \\ \hline
$1$ &  & $2^{dl}$ & $2^{dl+l}$ &  & $2^{dl+(d-1)l}$ \\ \hline
$2$ &  &  & $2^{2dl}$ &  & $2^{2dl+(d-2)l}$ \\ \hline
$\vdots$ & & &  & $\ddots$ & \\ \hline
$d$ &  &  &  &  & $2^{d\cdot dl}$\\ \hline
\end{tabular}$$

\bigskip

The only possible interference, for $l>l_1(P)$, is between $(i,j)=(0,d)$ and $(i,j)=(1,1)$ that both correspond to the general term $2^{dl}$. Each coefficient in front of $2^{idl+(j-i)l}$ does not depend on $l$ since $y$ will be chosen later to depend only on $P$ and the gap between any two distinct general terms in~\eqref{interferingsum} is at least $2^l$. The interfering term is the term in front of $2^{dl}$ in  the expansion of 
\begin{align*}
&(\alpha_0+\alpha_1(1+y2^l)+\cdots+(1+y2^l)^d)\, 2^{0\cdot dl}\\&+(\alpha_1+2\alpha_2(1+y2^l)+\cdots+d(1+y2^l)^{d-1})\,2^{1\cdot dl},
\end{align*}
since all the other general terms are at least of size $2^{2dl}$. Therefore, the interfering term is $y^d+\sum \limits_{1\leq i \leq d}i\alpha_i$.

On the other hand, by a similar calculation, we have \begin{align*}
P(1+y2^l+2^{dl+r})&=\sum \limits_{0\leq i\leq d} \left( \sum \limits_{i\leq j \leq d} \binom{j}{i}\alpha_j(1+y2^l)^{j-i} \right)2^{i(dl+r)}.
\end{align*}
We regroup terms by powers of $2$ as before. The general term is $2^{i(dl+r)+(j-i)l}$ and we have the following table:

$$\begin{tabular}{|c|c|c|c|c|c|c|}
\hline
\backslashbox{$i$}{$j$} & $0$ & $1$ & $2$ & $\cdots$ & $d$ \\ \hline
$0$ & $2^0$ & $2^{l}$ & $2^{2l}$ &  & $2^{dl}$ \\ \hline
$1$ &  & $2^{dl+r}$ & $2^{dl+l}2^r$ &  & $2^{dl+(d-1)l}2^r$ \\ \hline
$2$ &  &  & $2^{2dl}2^{2r}$ &  & $2^{2dl+(d-2)l}2^{2r}$ \\ \hline
$\vdots$ & & &  & $\ddots$ & \\ \hline
$d$ &  &  &  &  & $2^{d\cdot dl}2^{dr}$\\ \hline
\end{tabular}$$

\bigskip

Once again the only interference possible is for $(i,j)=(0,d)$ and $(i,j)=(1,1)$ for $l>l_2(P)$. Here, the interfering term is $y^d+2^r\sum \limits_{1 \leq i \leq d}i\alpha_i$. All the coefficients in the second table are identical to the ones given in the first table up to a multiplicative factor (a power of $2$). Since $t(2^{\mu}n)=t(n)$ for all $\mu \geq 0$, the contributions coming from the non-interfering terms are the same as in the former case. Put $z=\sum \limits_{1\leq i \leq d}i\alpha_i > 0$; we note that $z$ is a positive integer that only depends on $P$ and that $z\geq 2$. Summing up, for $l>l_3(P)$, we have 
\begin{align}\label{Goal} 
  t(P(1+y2^l+2^{dl}))&+t(P(1+y2^l+2^{dl+r}))\nonumber\\ 
  &\equiv t(y^d+z)+t(y^d+2^rz) \pmod 2.
\end{align}
Our final aim is to guarantee the existence of $r$ and $y$, only depending on $P$, such that right hand side of~\eqref{Goal} equals $1 \pmod 2$. Let $\lambda\geq 1$ be the unique integer with $2^{\lambda}\leq z< 2^{\lambda+1}$. Thus the most significant bit of $z$ is at position $2^{\lambda}$. Let $y=2^{\lambda}$, thus $z<2^{\lambda d}$, then we have \begin{align*}
t(y^d+z)= t(2^{\lambda d}+z)\equiv 1+t(z)\pmod 2.
\end{align*}
Let $r=\lambda d-\lambda\geq 1$, we have \begin{align*}
t(y^d+2^rz)= t(2^{\lambda d }+2^{\lambda d -\lambda}z)=t(2^{\lambda}+z)=t(z)
\end{align*}
since the most significant bit of $z$ is at position $2^{\lambda}$. We therefore conclude \begin{align*}
t(P(1+y2^l+2^{dl}))+t(P(1+y2^l+2^{dl+r}))\equiv 1+2\,t(z) \equiv 1 \pmod 2
\end{align*}
and we get property (ii).
\end{proof}

We have now all we need to prove \Cref{MainTheorem1}. 
\begin{proof}[Proof of \Cref{MainTheorem1}]

We first note that we can suppose $\alpha_i \geq 0$ for all $0\leq i<d-1$ and $\alpha_d=1$. Indeed, for positive integers $n,a$ we have $P(n+a)=\sum_{0 \leq i \leq d } \beta_i n^i$ with $\beta_i=\sum_{i\leq j \leq  d} \binom{j}{i} \alpha_j a^{j-i}$ and $\alpha_d=1$, such that for sufficiently large $a$, \begin{align*} \beta_i&=a^{d-i}\left( \binom{d}{i} \alpha_d+\sum \limits_{i\leq j < d}\binom{j}{i}\alpha_j a^{j-d} \right) \gg_{P} a^{d-i},
\end{align*}
and therefore we have $\beta_i \geq 0$ for all $i \geq 0$.
This translation by the positive integer $a$ that only depends on $P$ does not affect the measures of complexity that we study since we will suppose $N$ sufficiently large. Hence, without loss of generality, we can assume that all the coefficients of $P$ are positive integers.

Let $\alpha_{\max},z,\lambda,y,r$ be such as in the proof of the second part of \Cref{MainLemmaThueMorse}. Note that all these quantities only depend on $P$ and not on $l$. Let $N > N_0(P)$ be large enough and $M(\mathcal{T}_P,N)=M$. Let $l \geq 2$ be the integer defined by \begin{align*}
1+y2^l+2^{dl+r}<N \leq 1+y2^{l+1}+2^{d(l+1)+r}.
\end{align*} 
We follow the argument in the proof of Sun and Winterhof~\cite[Theorem 1]{sunwinterhof2019bis}. Assume that \begin{align*}
M< \frac{1}{2(2\alpha_{\max})^{1/d}}\,2^l,
\end{align*}
that is, there is a polynomial $f(x_1,\ldots,x_M)$ in $M$ variables with \begin{align}\label{cpx}
t(P(j+M))=f(t(P(j)),\ldots,t(P(j+M-1))),\quad j=0,1\ldots,N-M-1.
\end{align}
Note that for $0 \leq k \leq N-M-1$ the values of $t(P(k+M)),\ldots,t(P(N-1))$ are uniquely determined by the values of $t(P(k)),\ldots,t(P(k+M-1))$ by applying~\eqref{cpx} successively for $j=k,\ldots,N-M-1$. In particular, if \begin{align}\label{k1k2}
(t(P(k_1)),\ldots,t(P(k_1+M-1)))=(t(P(k_2)),\ldots,t(P(k_2+M-1)))
\end{align}
for some $k_1$ and $k_2$ with $0\leq k_1<k_2\leq N-M-1$, we get also \begin{align*}
(t(P(k_1+M)),\ldots,t(P(k_1+&N-k_2-1)))=\\&(t(P(k_2+M)),\ldots,t(P(N-1))).
\end{align*}
Take $k_1=2^{ld}$ and $k_2=2^{ld+r}$. By the first part of \Cref{MainLemmaThueMorse}, $(k_1,k_2)$ satisfies~\eqref{k1k2}. Then we have \begin{align*}
(t(P(2^{ld}+M)),\ldots,t(P(N+2^{ld}(1-2^{r})-1)))=\\(t(P(2^{ld+r}+M)),\ldots,t(P(N-1))).
\end{align*}
Since $N-1 \geq 1+y2^l+2^{dl+r}$ and $M\leq 1+y2^l$, this includes \begin{align*}
t(P(1+y2^l+2^{dl}))=t(P(1+y2^l+2^{dl+r})),
\end{align*}
which contradicts the second part of \Cref{MainLemmaThueMorse} and we get \begin{align}\label{eqmaxorder}
M \geq \frac{1}{2(2\alpha_{\max})^{1/d}}\,2^l \gg_P N^{1/d}. 
\end{align}
This finishes the proof of \Cref{MainTheorem1}.
\end{proof}

\section{Pattern sequences}

The identity~\eqref{trivial1} is not true in general for $s_k$ instead of $s_1$. For example, we have $s_2(4+2)=s_2(\tL \tL \tO)=1$ and $s_2(4)+s_2(2)=s_2(\tL \tO \tO)+s_2(\tL \tO)=0$. However, a very similar identity holds true when we add a $\tO$-bit in~\eqref{trivial1} between the expansions of $a2^{r}$ and $b$ : Let $a,b$ be positive integers and $0\leq b < 2^r$, then we have for all $k\geq 2$,\begin{align}\label{trivial2}
s_k(a2^{r+1}+b)=s_k(a)+s_k(b).
\end{align}

\begin{lem}\label{MainLemmaPattern}

Let $d\geq2$ and $P(X) \in \mathbb{Z}[X]$ with $P(X)=X^d+\alpha_{d-1}X^{d-1}+\cdots+\alpha_1X+\alpha_0$ such that $P(\mathbb{N}) \subset \mathbb{N}$ and all $\alpha_i\geq 0$. Put $\alpha_{\max}=\max(\alpha_i)$. Then there exists a positive integer $l_0(P,k)$ such that for all $l>l_0(P,k)$ the following two properties hold:

\begin{enumerate}
\item[(i)] For all $1\leq n < \frac{1}{4(2\alpha_{\max})^{1/d}}\;2^l$ and for all $s\geq 1$, $$p_k(P(n+2^{dl}))=p_k(P(n+2^{dl+s})).$$ 
\item[(ii)] There are nonnegative integers $y=y(P,k)$ and $s=s(P,k)$ such that \begin{align*}
p_k(P(1+y2^l+2^{dl}))\neq p_k(P(1+y2^l+2^{dl+s})).
\end{align*}
\end{enumerate}

\end{lem}

\begin{proof}
We can directly proceed as in the proof of \Cref{MainLemmaThueMorse} with~\eqref{trivial1} replaced by~\eqref{trivial2}. The ``digital gap'' in \eqref{trivial2} translates into a change by a factor $2$. We use the same notation as in \Cref{MainLemmaThueMorse}. For second part we get that for $l>l_0(P)$,
\begin{align*} p_k(P(1+y2^l+2^{dl}))&+p_k(P(1+y2^l+2^{dl+s}))\\ &\equiv p_k(y^d+z)+p_k(y^d+2^sz) \pmod 2.
\end{align*}
Let $f_a(x)=ax^3+ax^2-x+a$ be a polynomial where $a$ is a suitable positive integer that we will chose later. We write $$f_a(x)^d=\sum \limits_{0\leq i\leq 3d}\mu_ix^i$$ with $\mu_i \in \mathbb{Z}$. For $a>a_0(d)$, we have $\mu_i >0$ for $i \neq 1$ and $\mu_1=-da^{d-1}<0$. It follows that for $u>u_0(d)$ we have 
\begin{align} \label{f_a}
(f_a(2^u)^d)_2=\eta_1 \tO \ldots \tO \eta_2 \tO \ldots \tO \eta_{t-1} \tL \ldots \tL \eta_t \tO \ldots \tO \eta_{t+1}
\end{align} 
with $t=3d$ and some $\eta_i=\eta_i(a,d)$. For $i \in \{1,\ldots,t-2,t+1\}$, we can choose $\eta_i$ the binary expansion of $\mu_i$, $\eta_{t-1}$ such that its first bit is $\tL$ and its last bit is $\tO$ and $\eta_{t}$ such that its first bit is $\tO$ and its last bit is $\tL$. Note that the decomposition in $\eqref{f_a}$ is not unique in general. For $a=2^{\lambda}$ with $\lambda>\lambda_0(z)$ we have $p_k((f_a(2^{u})^d+z)=p_k((f_a(2^{u})^d)+p_k(z)$. Furthermore, the length of the $\tL$-block depends on $u$ and the transition from $u$ to $(u+1)$ adds exactly one $\tL$-bit to that $\tL$-block. Let $y=f_a(2^u)$ for $a=2^{\lambda} > \max(a_0,2^{\lambda_0})$ and $u>u_0(a,d)$. We write in the following $y^d=\omega_1\tO\ldots\tO\omega_2\tO\tL^{(\alpha)}\tO\omega_3$ where $\tL^{(\alpha)}$ is the block of $\tL$-bits in $\eqref{f_a}$ between $\eta_{t-1}$ and $\eta_t$. For $u>u_1(a,d,z,k)$ sufficiently large, we can ensure that $\alpha > \max(\lceil \log_2(z) \rceil , k)$. We write $z=\tL^{(i)}\tO\omega'$ or $z=\tL^{(i)}$ for some $i\geq 1$ and digital block $\omega'$ which can be possibly empty. We choose $s$ in a way that the first $i$ $\tL$-bits interfere with the last $i$ $\tL$-bits of the inner $\tL$-bits block of $y^d$. Note that this is always possible since $\alpha$ is larger than the length of $z$. When $z=\tL^{(i)}\tO\omega'$ we get 

\bigskip

\begin{tabular}{ccccccll}
& $\omega_1\tO\ldots\tO$&$ w_2$&$\overbrace{\tL \ldots \tL}^{\alpha-i}$ && $\overbrace{\tL\ldots\tL\tL}^{i}$ &$\tO\omega_3$&$=y^d$\\
+ & & & & & $\tL\ldots\tL\tL$ & $\tO\omega'$ & $=2^sz$\\
\hline
& $\omega_1\tO\ldots\tO$&$ (w_2+1)$& $\tO\ldots\tO$ && $\tL\ldots\tL\tO$ &$\tO(\omega_3+\omega')$& $=y^d+2^sz$
\end{tabular}  

\bigskip

When $z=\tL^{(i)}$ we get

\begin{tabular}{ccccccll}
& $\omega_1\tO\ldots\tO$&$ w_2$&$\overbrace{\tL \ldots \tL}^{\alpha-i}$ && $\overbrace{\tL\ldots\tL\tL}^{i}$ &$\tO\omega_3$&$=y^d$\\
+ & & & & & $\tL\ldots\tL\tL$ & $\tO \cdots \tO$ & $=2^sz$\\
\hline
& $\omega_1\tO\ldots\tO$&$ (w_2+1)$&$\tO\ldots\tO$ &  & $\tL\ldots\tL\tO$ &$\tO\omega_3$& $=y^d+2^sz$
\end{tabular}

\bigskip

In both cases, for $u>u_2(a,d,z,k)$ sufficiently large, $p_k(y^d+2^sz)$ is constant with respect to $u$ since no more $k$-pattern is created or canceled. In fact, only the length of the $\tO$-blocks and the one of the $\tL$-block change with $u$ in $\eqref{f_a}$ but the sum $y^d+2^sz$ reduces the $\tL$-block of length $\alpha$ to a $\tL$-block of length $(i-1)$ which does not depend on $u$ anymore. Furthermore, the number of $k$-patterns in $\omega_1,\omega_2,\omega_3$ and $\omega'$ are independent from $u$. Hence $p_k(y^d+2^sz)$ is constant for all $u>u_2(a,d,z,k)$. What matters now is the number of $k$-patterns we cancel in $y^d$ by adding $2^sz$. Since $\alpha>k$, the transition from $u$ to $(u+1)$ implies that if we add $2^sz$ to $y^d$, we cancel one more $k$-pattern in $y^d$. Thus we can choose $u>u_2(a,d,z,k)$ in such a way that the parity of the number of $k$-patterns that are canceled in $y^d$ by adding $2^sz$ changes. This leads to a solution of $p_k(y^d+2^sz)\equiv p_k(y^d)+p_k(z)+1\pmod 2$ and the lemma is proved.

\end{proof}

\begin{proof}[Proof of \Cref{MainTheorem2}] 
The proof follows the lines of the proof of \Cref{MainTheorem1} where we replace \Cref{MainLemmaThueMorse} by \Cref{MainLemmaPattern}. We note that the size of $y$ only occurs in the final step of the proof, namely, inequality~\eqref{eqmaxorder}. Thus we have 
\begin{align*}
M \geq \frac{1}{4(2\alpha_{\max})^{1/d}}\,2^l \gg_{P,k} N^{1/d}, 
\end{align*}
which proves \Cref{MainTheorem2}.
\end{proof}

\section{Open problems}

We can define a larger family of automatic sequences on $\{0,\ldots,q-1\}$ for $q\geq 2$ by the following procedure. Let $\omega \in \{0,\ldots,q-1\}^k \setminus\{0\ldots0\}$ be a pattern of length $k$ and $e_{\omega}(n)$ be the number of occurrences of $\omega$ in the $q$-ary representation of $n$. We define the sequence $(\rho(n))_n$ by
\begin{align}\label{generalpattern}
  \rho(n)\equiv e_{\omega}(n) \pmod m,
\end{align}
where $m\geq 2$ is a fixed integer. From our reasoning in the proofs, one can extend our result for $\omega=(q-1)\ldots(q-1)$ and any $m\geq 2$. Indeed in the proof of \Cref{MainLemmaPattern}, we use a $\tL$-bit carry propagation along a large block of $\tL$-bits. The same argument works to cancel bits of digits $(q-1)$ in base-$q$ expansions. 

\begin{pb}
Extend \Cref{MainTheorem2} to general pattern sequences as defined in \eqref{generalpattern}. It is essential to our proof that the pattern must appear at least once in $y^d$. For prime $m$ and $d$ such $m\nmid d$, Hensel's lifting lemma can be useful to make sure that it appears. It seems that other ideas are needed in the case $m \mid d$. 
\end{pb}

\begin{pb}[\cite{drmotamauduitrivat2019}, Conjecture 1]
Show that the subsequences of the Thue--Morse sequence along any polynomial of degree $d\geq 3$ are normal. A lower bound on the subword complexity was established by Moshe~\cite[Corollary 3]{moshe2007} and~\cite{drmotamauduitrivat2011} gives a partial answer for generalized Thue--Morse sequence with large $q$.
\end{pb}

\begin{pb}[\cite{sunwinterhof2019bis}]
Prove lower bounds on the expansion complexity of the Thue--Morse sequence along polynomials of degree $d\geq 2$. This question seems to be hard even for small degree polynomials.
\end{pb}

\noindent {\bf Acknowledgements.} This work was supported partly by the french PIA project `` Lorraine Universit\'{e} d'Excellence '', reference ANR-15-IDEX-04-LUE. The author expresses his gratitude to D. Jamet and T. Stoll for the supervision of this work, and A. Winterhof for the right path to the proof.

\bibliographystyle{amsplain}
\bibliography{biblio}

\end{document}